\documentclass[12pt]{article}

\usepackage{amsmath, amsthm, amssymb}
\usepackage{geometry}
\newtheorem{theorem}{Theorem}
\newtheorem{lemma}[theorem]{Lemma}
\newtheorem{proposition}[theorem]{Proposition}
\newtheorem{corollary}[theorem]{Corollary}

\newtheorem{remark}[theorem]{Remark}

\title{Exact Magnetization Levels and Pathwise Confinement in Stochastic Ising Models}

\author{
Jean-Gabriel~Attali\thanks{Contact: jean-gabriel.attali@devinci.fr} \\
\small De Vinci Higher Education, De Vinci Research Center, Paris, France
}

\date{}

\begin{document} 
\maketitle

\begin{abstract}

We study pathwise confinement in exact magnetization levels for
zero-field heat-bath Glauber dynamics. For
$C_r=\{\sigma:M_n(\sigma)\to r\}$, deterministic-time confinement is
shown to be equivalent to pathwise confinement. We also construct a
dense family of periodic antisymmetric configurations whose trajectories
remain in $C_0$ for all time almost surely.

For the two-dimensional Ising model with $\beta>\beta_c$, the levels
$C_{\pm m_\beta}$ contain pathwise cores of full $\mu^\pm$-measure.
The antisymmetric configurations remain outside these two equilibrium
levels, while their Ces\`aro averages converge to
$\frac12(\mu^++\mu^-)$. Thus pathwise confinement is richer than the
Gibbs-phase classification, without producing a third Gibbs phase.

\end{abstract}

\section{Introduction}

Consider the two-dimensional nearest-neighbor Ising model at zero external
field and inverse temperature $\beta>\beta_c$, in the phase-coexistence
regime. The model then has two extremal Gibbs states, $\mu^+$ and $\mu^-$,
with spontaneous magnetizations $+m_\beta$ and $-m_\beta$. By the
Aizenman--Higuchi theorem, every Gibbs state is a convex combination of
these two phases
\cite{Aizenman80,Higuchi81,Georgii11,FriedliVelenik17}.

We study the associated infinite-volume continuous-time heat-bath Glauber
dynamics on $\Omega=\{-1,+1\}^{\mathbb Z^2}$, constructed through the
Harris graphical representation \cite{Harris72,Liggett85}. For the boxes
$\Lambda_n=[-n,n]^2\cap\mathbb Z^2$, let $M_n(\sigma)$ denote the empirical
magnetization in $\Lambda_n$, and define
$C_r=\{\sigma:M_n(\sigma)\to r\}$. Translation ergodicity gives
$\mu^+(C_{m_\beta})=\mu^-(C_{-m_\beta})=1$. We write
\[
R=\Omega\setminus\bigl(C_{m_\beta}\cup C_{-m_\beta}\bigr),
\]
so that $\mu^+(R)=\mu^-(R)=0$.

The main point of the paper is that the pathwise confinement geometry of
the dynamics is richer than the equilibrium Gibbs-phase classification
alone suggests. For an exact magnetization level $C_r$, let $D_r$ be the
set of configurations whose distribution remains in $C_r$ at every
deterministic time, and let $S_{C_r}$ be the maximal set of configurations
whose entire trajectory remains in $C_r$ almost surely. We prove the
structural identity
\[
D_r=S_{C_r}.
\]
The nontrivial implication follows from a vanishing-mesh clock argument:
between two nearby mesh times, only sites whose Poisson clocks ring can
change, and their spatial density tends to zero with the mesh size. This
argument depends only on the local-clock structure of the dynamics and is
not specific to the phase-coexistence regime.

Applied to the equilibrium magnetization levels, this identity yields
pathwise cores of full equilibrium measure:
\[
\mu^+(S_{C_{m_\beta}})=1,
\qquad
\mu^-(S_{C_{-m_\beta}})=1.
\]
Thus a typical initial configuration under either extremal phase remains
in the corresponding exact magnetization level for all times almost
surely.

We also construct a dense family
$\mathcal A_{\mathrm{anti}}$ of periodic antisymmetric configurations.
Their translation--spin-flip symmetry is preserved in law by the
zero-field dynamics and forces fixed-time centering. Consequently,
\[
\mathcal A_{\mathrm{anti}}
\subset S_{C_0}
\subset S_R.
\]
Hence both the centered pathwise core and the pathwise core of $R$ are
nonempty, dense, and semigroup-absorbing in the universally completed
measurable structure.

This does not contradict the equilibrium phase picture. For every
$\eta\in\mathcal A_{\mathrm{anti}}$, the Ces\`aro averages
$\frac1T\int_0^T\delta_\eta P_t\,dt$ converge to
$\frac12(\mu^++\mu^-)$. The centered confinement is therefore a pathwise
phenomenon and does not produce a third Gibbs phase.

Section~2 introduces maximal pathwise confinement sets. Section~3
constructs $\mathcal A_{\mathrm{anti}}$ and proves fixed-time centering.
Section~4 establishes the fixed-time-to-pathwise upgrade and derives the
equilibrium and centered pathwise cores. Section~5 identifies the
long-time Ces\`aro limit.

\section{Exact sectors and pathwise confinement}
\label{sec:exact-sectors}

Throughout the phase-coexistence results, we assume
$\beta>\beta_c$ and write
\[
R
=
\Omega\setminus
\bigl(C_{m_\beta}\cup C_{-m_\beta}\bigr).
\]
By translation ergodicity,
\[
\mu^+(C_{m_\beta})=1,
\qquad
\mu^-(C_{-m_\beta})=1.
\]
Hence
\[
\mu^+(R)=\mu^-(R)=0.
\]
Since $m_\beta>0$, one also has
\[
C_0\subset R.
\]

\subsection{Topology of the equilibrium magnetization levels}

We record the Baire-category size of the equilibrium magnetization
levels. Their complement is topologically large, although it is null
under both extremal Gibbs phases.

\begin{proposition}[Topology of the equilibrium magnetization levels]
Assume $\beta>\beta_c$. Then $C_{m_\beta}$ and $C_{-m_\beta}$ are
meagre in $\Omega$. Consequently,
\[
R=\Omega\setminus\bigl(C_{m_\beta}\cup C_{-m_\beta}\bigr)
\]
is comeagre.
\end{proposition}

\begin{proof}
Fix $r\in(-1,1)$ and choose
$0<\varepsilon<\min\{1-r,1+r\}$. For $N\geq1$, set
\[
F_N
=
\bigcap_{n\geq N}
\{\sigma:|M_n(\sigma)-r|\leq\varepsilon\}.
\]
Then
\[
C_r\subset\bigcup_{N\geq1}F_N.
\]
Each $F_N$ is closed and has empty interior: every nonempty cylinder
contains a configuration equal to $+1$ outside a finite box, for which
$M_n(\sigma)\to1$. Thus $C_r$ is meagre. Applying this to
$r=\pm m_\beta$ gives the result.
\end{proof}

\subsection{Maximal pathwise confinement sets}

Absorbing and invariant regions are classical objects in Markov process
theory \cite{Rosenblatt71}. In the present infinite-volume setting we use the
following pathwise confinement sets.

Let \(D([0,\infty),\Omega)\) denote the Skorokhod space of càdlàg paths with
values in \(\Omega\). For a Borel set \(C\subset\Omega\), define the path event
\[
        H_C
        =
        \{\omega\in D([0,\infty),\Omega):\omega(t)\in C
        \text{ for all }t\ge0\}.
\] 
The maximal pathwise confinement set in \(C\) is
\[
        S_C
        =
        C\cap
        \{\xi\in\Omega:\mathbb P_\xi(H_C)=1\}.
\]
All probabilities of universally measurable sets are understood with respect to
the universal completion of the Borel \(\sigma\)-field.

Thus \(S_C\) is the set of initial configurations in \(C\) whose trajectories
remain in \(C\) for all future times, almost surely.

We shall use this definition for exact levels such as
$C_0$, $C_{m_\beta}$, and $C_{-m_\beta}$, and also for their complement $R$.
In particular,
\[
        S_{C_0}
        =
        C_0\cap
        \{\xi:\mathbb P_\xi(\sigma_t\in C_0 \ \forall t\ge0)=1\},
\]
and
\[
        S_R
        =
        R\cap
        \{\xi:\mathbb P_\xi(\sigma_t\in R \ \forall t\ge0)=1\}.
\]
Since \(C_0\subset R\), one has immediately
\[
        S_{C_0}\subset S_R.
\]

The following elementary fact records the measurability and absorption properties of the maximal confinement set.

\begin{proposition}[Maximal pathwise confinement sets]
\label{prop:maximal-pathwise-confinement}
Let \(C\subset\Omega\) be Borel. Then \(S_C\) is universally measurable and
semigroup-absorbing in the universally completed sense. Namely, for every
\(\xi\in S_C\) and every \(u\ge0\),
\[
        P_u(\xi,S_C)=1.
\]
\end{proposition}

\begin{proof}
The exit event
\[
H_C^c
=
\{\omega:\exists t\geq0,\ \omega(t)\notin C\}
\]
is the projection onto path space of the Borel set
\[
\{(t,\omega)\in[0,\infty)\times D([0,\infty),\Omega):
\omega(t)\notin C\}.
\]
Hence $H_C^c$ is analytic, so $H_C$ is universally measurable. It follows that
\[
S_C=C\cap\{\xi:P_\xi(H_C)=1\}
\]
is universally measurable.

Now let $\xi\in S_C$ and $u\geq0$. Since $P_\xi(H_C)=1$, one has
$\sigma_u\in C$ almost surely. By the Markov property,
\[
P_{\sigma_u}(H_C)=1
\qquad P_\xi\text{-a.s.}
\]
Therefore $\sigma_u\in S_C$ almost surely, and hence
\[
P_u(\xi,S_C)=1.
\]
\end{proof}

\section{Periodic antisymmetry and fixed-time centering}
\label{sec:periodic-antisymmetry}

We now construct the deterministic configurations which generate the centered
part of the confinement structure. The construction uses periodic configurations
satisfying a translation-spin-flip symmetry on a finite-index quotient of
\(\mathbb Z^2\). This symmetry is preserved in law by the zero-field dynamics
and forces centering at every deterministic time.

\subsection{Antisymmetric periodicity}

Let \(\Theta\) denote the global spin-flip,
\[
        (\Theta\sigma)_i=-\sigma_i,
\]
and let \(\tau_u\) denote translation by \(u\in\mathbb Z^2\), with the
convention
\[
        (\tau_u\sigma)_i=\sigma_{i-u}.
\]

A configuration \(\eta\in\Omega\) is called periodic antisymmetric if there
exist a finite-index subgroup \(L\le \mathbb Z^2\) and a vector
\(u\in\mathbb Z^2\) such that
\[
        \tau_\ell\eta=\eta
        \qquad \forall \ell\in L,
\]
and
\[
        \tau_u\eta=\Theta\eta .
\]
We denote the set of all such configurations by
$\mathcal A_{\mathrm{anti}}.$
Thus \(\mathcal A_{\mathrm{anti}}\) consists of configurations which are
periodic modulo a finite-index subgroup and whose translation by some lattice
vector is their global spin-flip.

\subsection{Fixed-time centering}

\begin{lemma}[Fixed-time centering]
\label{lem:fixed-time-centering}
For every \(\eta\in\mathcal A_{\mathrm{anti}}\) and every \(t\ge0\),
\[
        \delta_\eta P_t(C_0)=1.
\]
\end{lemma}

\begin{proof}
Fix \(\eta\in\mathcal A_{\mathrm{anti}}\). Choose a finite-index subgroup
\(L\subset\mathbb Z^2\) and \(u\in\mathbb Z^2\) such that
\[
        \tau_\ell\eta=\eta\quad(\ell\in L),
        \qquad
        \tau_u\eta=\Theta\eta .
\]
Fix \(t\ge0\), and write
\[
        \nu_t:=\delta_\eta P_t .
\]

We first record a consequence of the graphical construction. Let \(H_t\) denote the
marked Harris noise in the time interval \([0,t]\). The field \(H_t\) is ergodic under
the spatial action of \(L\). Moreover the time-\(t\) configuration can be written as
\[
        \sigma_t=\Phi_t(\eta,H_t).
\]
We use the standard graphical construction on the full space-time Poisson field; for
each fixed \(t\), the map \(H_t\mapsto\Phi_t(\eta,H_t)\) is measurable and commutes
with spatial translations. This follows from the usual finite backward-cluster
construction for finite-range spin-flip dynamics.

Since \(\eta\) is \(L\)-periodic, for every \(\ell\in L\),
\[
        \Phi_t(\eta,\tau_\ell H_t)
        =
        \tau_\ell\Phi_t(\eta,H_t).
\]
Thus \(\nu_t\) is the image of an \(L\)-ergodic system by an \(L\)-equivariant factor.
Hence \(\nu_t\) is \(L\)-invariant and \(L\)-ergodic.

Let
\[
        Q:=\mathbb Z^2/L.
\]
For \(a\in Q\), choose a representative, still denoted \(a\), and set
\[
        \Gamma_a:=a+L,
        \qquad
        c_a:=\int_\Omega \sigma_a\,\nu_t(d\sigma).
\]
The number \(c_a\) is well-defined on \(Q\), since \(\nu_t\) is
\(L\)-invariant.

For each \(a\in Q\), put
\[
        F_n^a:=\{\ell\in L:\ a+\ell\in\Lambda_n\}.
\]
The sets \(F_n^a\) form a Følner sequence in \(L\). By the pointwise ergodic
theorem for amenable group actions \cite{Lindenstrauss01}, for
\(\nu_t\)-almost every \(\sigma\),
\[
        \frac{1}{|\Lambda_n\cap \Gamma_a|}
        \sum_{i\in\Lambda_n\cap \Gamma_a}\sigma_i
        \longrightarrow c_a .
\]
Since \(Q\) is finite, this convergence holds simultaneously for all
\(a\in Q\) on a common \(\nu_t\)-full set.

We now use the antisymmetry. The zero-field dynamics commutes with
translations and with the global spin-flip. Hence
\[
        (\tau_u)_*\nu_t
        =
        \delta_{\tau_u\eta}P_t
        =
        \delta_{\Theta\eta}P_t
        =
        \Theta_*\nu_t .
\]
With our convention \((\tau_v\sigma)_i=\sigma_{i-v}\), applying this identity
to the observable \(\sigma_{a+u}\) gives
\[
        c_a=-c_{a+u},
        \qquad a\in Q.
\]

Let \(T:Q\to Q\) be the permutation \(T(a)=a+u\). On each cycle
\[
        a,\ T(a),\ldots,T^{k-1}(a),
\]
the relation \(c_{T(b)}=-c_b\) gives
\[
        c_{T^j(a)}=(-1)^j c_a,
        \qquad 0\le j\le k-1.
\]
If \(k\) is even, the terms cancel pairwise. If \(k\) is odd, then
\(T^k(a)=a\) forces \(c_a=-c_a\), hence all entries on the cycle vanish.
In both cases, the sum of the \(c_b\)'s over the cycle is zero. Summing over
the cycles of \(T\), we obtain
\[
        \sum_{a\in Q}c_a=0.
\]

Finally, for \(\nu_t\)-almost every \(\sigma\), decompose the empirical
magnetization over the cosets of \(L\):
\[
        M_n(\sigma)
        =
        \sum_{a\in Q}
        \frac{|\Lambda_n\cap \Gamma_a|}{|\Lambda_n|}
        \left(
        \frac{1}{|\Lambda_n\cap \Gamma_a|}
        \sum_{i\in\Lambda_n\cap \Gamma_a}\sigma_i
        \right).
\]
Since \(L\) has finite index,
\[
        \frac{|\Lambda_n\cap \Gamma_a|}{|\Lambda_n|}
        \longrightarrow
        \frac1{|Q|}
        \qquad(a\in Q).
\]
Using the coset ergodic limits above, we obtain
\[
        M_n(\sigma)
        \longrightarrow
        \frac1{|Q|}
        \sum_{a\in Q}c_a
        =
        0.
\]
Thus \(\nu_t(C_0)=1\), i.e.
\[
        \delta_\eta P_t(C_0)=1 .
\]
\end{proof}

\subsection{Density of the generating configurations}

\begin{lemma}[Density of periodic antisymmetric configurations]
\label{lem:anti-dense}
The set \(\mathcal A_{\mathrm{anti}}\) is dense in \(\Omega\) for the
product topology.
\end{lemma}

\begin{proof}
It is enough to show that every nonempty cylinder set intersects
\(\mathcal A_{\mathrm{anti}}\). Let \(U\subset\Omega\) be a nonempty cylinder
set. Then there exist a finite set \(F\Subset\mathbb Z^2\) and a pattern
\(\xi_F\in\{-1,+1\}^F\) such that
\[
        U=\{\sigma\in\Omega:\sigma|_F=\xi_F\}.
\]

Choose \(N\ge1\) so large that
\[
        F\subset[-N,N]^2.
\]
Set
\[
        u=(2N+1,0),
        \qquad
        L=(4N+2)\mathbb Z\times(4N+2)\mathbb Z.
\]
Then \(L\) is a finite-index subgroup of \(\mathbb Z^2\), and
\[
        2u\in L,
        \qquad
        u\notin L.
\]
Moreover, the sets
\[
        F+\ell,
        \qquad
        F+u+\ell,
        \qquad \ell\in L,
\]
are pairwise disjoint, except for the trivial coincidences within the same
translate. Indeed, if two such translates intersect, then some element of
\[
        F-F:=\{i-j:i,j\in F\}
        \subset[-2N,2N]^2
\]
would be equal either to a nonzero element of \(L\), or to an element of
\(u+L\), or to an element of \(-u+L\). This is impossible by the definitions of
\(u\) and \(L\).

Let
\[
        Q:=\mathbb Z^2/L.
\]
Since \(2u\in L\), the map
\[
        T:Q\to Q,
        \qquad
        T(a)=a+u,
\]
is an involution. Since \(u\notin L\), it has no fixed point.

We now prescribe \(\eta\) on the \(L\)-periodic copies of \(F\) and \(F+u\) by
setting
\[
        \eta_{i+\ell}=\xi_F(i),
        \qquad
        \eta_{i+u+\ell}=-\xi_F(i),
        \qquad i\in F,\ \ell\in L.
\]
The disjointness proved above ensures that these prescriptions are consistent.

On the remaining \(T\)-orbits in \(Q\), choose one representative \(a\) from
each pair \(\{a,a+u\}\), assign \(\eta_a\in\{-1,+1\}\) arbitrarily, and impose
\[
        \eta_{a+u}=-\eta_a.
\]
Extend the resulting assignment \(L\)-periodically to all of \(\mathbb Z^2\).

By construction,
\[
        \tau_\ell\eta=\eta
        \qquad \forall \ell\in L.
\]
Moreover, the values on each pair of cosets \(a\) and \(a+u\) are opposite.
With the convention \((\tau_u\eta)_i=\eta_{i-u}\), this gives
\[
        \tau_u\eta=\Theta\eta.
\]
Thus \(\eta\in\mathcal A_{\mathrm{anti}}\). Finally,
\[
        \eta|_F=\xi_F,
\]
so \(\eta\in U\). Hence every nonempty cylinder set meets
\(\mathcal A_{\mathrm{anti}}\), and \(\mathcal A_{\mathrm{anti}}\) is dense in
\(\Omega\).
\end{proof}

Every $\eta\in\mathcal A_{\mathrm{anti}}$ belongs to $C_0$.
Indeed, the translation--spin-flip symmetry pairs opposite spins in
each period cell, whose average is therefore zero. Hence
$M_n(\eta)\to0$. Combining this observation with Lemma~\ref{lem:fixed-time-centering} gives
\[
\mathcal A_{\mathrm{anti}}\subset D_0,
\qquad
D_0
=
\{\xi\in C_0:
P_t(\xi,C_0)=1\text{ for every }t\geq0\}.
\]
By Lemma~\ref{lem:anti-dense}, $D_0$ is dense in $\Omega$.

\section{From deterministic-time to pathwise confinement}
\label{sec:fixed-time-to-pathwise}

The next argument upgrades deterministic-time confinement in exact
magnetization levels to pathwise confinement. The mechanism is a
vanishing-mesh clock estimate: between two nearby mesh times, only the sites
whose Poisson clocks ring can change, and their spatial density vanishes with
the mesh size.

\subsection{A vanishing-mesh equivalence for exact levels}

For exact magnetization levels, deterministic-time confinement is equivalent
to pathwise confinement. This statement is structural and does not assert that
the corresponding sets are nonempty.

For \(r\in[-1,1]\), recall
\[
        C_r=\{\sigma\in\Omega:\ M_n(\sigma)\to r\}.
\]
The pathwise confinement set in \(C_r\) is
\[
        S_{C_r}
        :=
        \{\xi\in C_r:\ \mathbb P_\xi(\sigma_t\in C_r\ \forall t\ge0)=1\}.
\]
We also introduce the deterministic-time confinement set
\[
        D_r
        :=
        \{\xi\in C_r:\ P_t(\xi,C_r)=1\ \text{for every }t\ge0\}.
\]
Thus \(S_{C_r}\) records confinement in \(C_r\) along the whole path, whereas
\(D_r\) records confinement at every deterministic time. Clearly
\[
        S_{C_r}\subset D_r.
\]
The next proposition proves the converse inclusion.

\begin{proposition}[Fixed-time and pathwise confinement coincide]
\label{prop:vanishing-mesh-upgrade}
For every \(r\in[-1,1]\),
\[
        D_r=S_{C_r}.
\]
\end{proposition}

\begin{proof} Fix \(\xi\in D_r\).

Let \(\delta>0\). For \(k\ge0\) and \(i\in\mathbb Z^2\), let
\[
        J_i^{(k,\delta)}
        :=
        \mathbf 1\{\text{the clock at site }i
        \text{ rings during }[k\delta,(k+1)\delta]\}.
\]
Then \((J_i^{(k,\delta)})_{i\in\mathbb Z^2}\) is an i.i.d. Bernoulli field
with parameter
\[
        p_\delta:=1-e^{-\delta}.
\]
By the strong law of large numbers, for every fixed \(k\),
\[
\frac{1}{|\Lambda_n|}
\sum_{i\in\Lambda_n}
J_i^{(k,\delta)}
\longrightarrow p_\delta
\]
almost surely. Since there are only countably many \(k\)'s, this convergence holds
simultaneously for all \(k \ge 0\) on an event of probability one. We shall use this
below with \(k\) chosen as the mesh index corresponding to a given time \(t\).

Since \(\xi\in D_r\), for every \(k\ge0\),
\[
        \mathbb P_\xi(\sigma_{k\delta}\in C_r)=1.
\]
Intersecting again over the countable set of mesh times, we may also assume
that
\[
        \sigma_{k\delta}\in C_r
        \qquad\forall k\ge0.
\]

Fix a graphical realization in this probability-one event, and let \(t\ge0\).
If \(t=k\delta\) for some \(k\), then \(\sigma_t\in C_r\) by construction.
Otherwise choose \(k\ge0\) such that
\[
        t\in(k\delta,(k+1)\delta).
\]
If the spin at site \(i\) changes between \(k\delta\) and \(t\), then the
clock at \(i\) must have rung during \([k\delta,(k+1)\delta]\). Hence
\[
        |\sigma_t(i)-\sigma_{k\delta}(i)|
        \le
        2J_i^{(k,\delta)}.
\]
Averaging over \(\Lambda_n\), we obtain
\[
        |M_n(\sigma_t)-M_n(\sigma_{k\delta})|
        \le
        2\frac1{|\Lambda_n|}
        \sum_{i\in\Lambda_n}J_i^{(k,\delta)}.
\]
Since \(\sigma_{k\delta}\in C_r\), we have
\[
        M_n(\sigma_{k\delta})\longrightarrow r.
\]
Therefore
\[
        \limsup_{n\to\infty}
        |M_n(\sigma_t)-r|
        \le
        2p_\delta
        =
        2(1-e^{-\delta}).
\]

All these events are understood on the same Harris graphical probability space, which
carries the clocks for all sites and all time intervals simultaneously. Fix a deterministic
sequence \(\delta_m \downarrow 0\). For each fixed \(m\), the event constructed above has
probability one and, on this event, the estimate holds simultaneously for every
\(k\ge0\) and every \(t\in[k\delta_m,(k+1)\delta_m]\). Thus no intersection over the
uncountable set of times is involved.

Intersecting these probability-one events over \(m\ge1\), we obtain a single
probability-one event on which, for every \(m\ge1\) and every \(t\ge0\),
\[
\limsup_{n\to\infty}
|M_n(\sigma_t)-r|
\le 2(1-e^{-\delta_m}).
\]
Since \(\delta_m \downarrow 0\), the right-hand side tends to zero. Hence, on this
single event, for every \(t\ge0\),
\[
\limsup_{n\to\infty}
|M_n(\sigma_t)-r|=0.
\]
Therefore, for every \(t\ge0\),
\[
        \limsup_{n\to\infty}|M_n(\sigma_t)-r|=0,
\]
hence \(M_n(\sigma_t)\to r\).

This proves
\[
        \mathbb P_\xi(\sigma_t\in C_r\ \forall t\ge0)=1,
\]
so that \(\xi\in S_{C_r}\). Thus \(D_r\subset S_{C_r}\) hence \(D_r= S_{C_r}\).
\end{proof}

\begin{remark}
The fixed-time-to-all-time upgrading mechanism is related in spirit to
the arguments used in Propositions~1.1 and~2.1 of \cite{Steif03}. The present setting is different
in that the relevant observable is an infinite-volume empirical
average. The upgrade follows from the fact that, during a time interval
of length $\delta$, only a spatial density $1-e^{-\delta}$ of sites
have their clocks ring, and this density vanishes as
$\delta\downarrow0$.
\end{remark}

\begin{remark}[Extension to local observables]
The proof of Proposition~\ref{prop:vanishing-mesh-upgrade} also applies to empirical averages of any
bounded local observable. Indeed, such an observable can change during
an interval of length $\delta$ only if a clock rings in its finite
dependence set, and the spatial density of affected translates tends
to zero as $\delta\downarrow0$.
\end{remark}

We now apply Proposition~ \ref{prop:vanishing-mesh-upgrade} first to the equilibrium magnetization
levels and then to the centered level.

\subsection{Full-measure cores in the equilibrium magnetization levels}

The preceding equivalence applies in particular to the equilibrium
magnetization levels 
\(C_{m_\beta}\) and \(C_{-m_\beta}\). Since the extremal Gibbs phases are invariant and give full mass to
these levels, the corresponding pathwise cores have full equilibrium measure.

\begin{corollary}[Full-measure pathwise cores in the equilibrium
magnetization levels]
\label{cor:pure-full-measure-cores}
Assume $\beta>\beta_c$. Then
\[
        \mu^+(S_{C_{m_\beta}})=1,
        \qquad
        \mu^-(S_{C_{-m_\beta}})=1.
\]
Moreover, \(S_{C_{m_\beta}}\) and \(S_{C_{-m_\beta}}\) are dense in
\(\Omega\).
\end{corollary}

\begin{proof} Since
\(\mu^+(C_{m_\beta})=1\) and \(\mu^+\) is invariant for the dynamics, for
every \(t\ge0\),
\[
        \int P_t(\xi,C_{m_\beta})\,\mu^+(d\xi)
        =
        \mu^+(C_{m_\beta})
        =
        1.
\]
Hence, for every fixed \(t\ge0\),
\[
        P_t(\xi,C_{m_\beta})=1
        \quad\text{for \(\mu^+\)-almost every }\xi.
\]
Intersecting over \(t\in\mathbb Q_+\), we obtain a set of full
\(\mu^+\)-measure on which
\[
        P_t(\xi,C_{m_\beta})=1
        \qquad \forall t\in\mathbb Q_+.
\]
The proof of Proposition~\ref{prop:vanishing-mesh-upgrade} only uses
deterministic-time confinement at the mesh times. Choosing a sequence of
rational mesh sizes \(\delta_m\downarrow0\), all mesh times \(k\delta_m\)
are rational. Hence the same vanishing-mesh argument applies from the above
rational-time confinement and yields
\[
        \xi\in S_{C_{m_\beta}}.
\]
Therefore
\[
        \mu^+(S_{C_{m_\beta}})=1.
\]

It remains to prove density. Since the plus phase has full support on
\(\Omega\), every nonempty open set has positive \(\mu^+\)-measure. Hence
no universally measurable set of full \(\mu^+\)-measure can be disjoint from
a nonempty open set. Therefore \(S_{C_{m_\beta}}\) is dense in \(\Omega\).
The same argument applied to \(\mu^-\) gives
\[
        \mu^-(S_{C_{-m_\beta}})=1
\]
and the density of \(S_{C_{-m_\beta}}\).
\end{proof}

 \subsection{The centered and complementary pathwise cores}

\begin{corollary}[Centered and complementary pathwise confinement]
Assume $\beta>\beta_c$. Then
\[
\mathcal A_{\mathrm{anti}}
\subset S_{C_0}
\subset S_R.
\]
In particular, $S_{C_0}$ and $S_R$ are dense in $\Omega$ and
semigroup-absorbing in the universally completed sense.
\end{corollary}

\begin{proof}
By the conclusion of Section~3,
\[
\mathcal A_{\mathrm{anti}}\subset D_0.
\]
Proposition~4.1 gives $D_0=S_{C_0}$, and hence
\[
\mathcal A_{\mathrm{anti}}\subset S_{C_0}.
\]
Since $C_0\subset R$, every trajectory confined to $C_0$ is also
confined to $R$, so $S_{C_0}\subset S_R$. Density follows from the
density of $\mathcal A_{\mathrm{anti}}$, and the semigroup-absorbing
property follows from Proposition~2.2.
\end{proof}

\section{No third Gibbs phase}

Assume $\beta>\beta_c$ and let
$\eta\in\mathcal A_{\mathrm{anti}}$. Set
\[
\mu_T
=
\frac{1}{T}\int_0^T\delta_\eta P_t\,dt.
\] 
Since \(\Omega\) is compact, the family \((\mu_T)_{T>0}\) is relatively
compact for the weak topology. By the usual Krylov--Bogoliubov argument,
every subsequential limit is invariant for the Glauber semigroup. In
dimension two, the Holley--Stroock theorem \cite{HolleyStroock77} and the
Aizenman--Higuchi classification \cite{Aizenman80,Higuchi81} imply that any
such limit is of the form
\[
        \lambda\mu^+ +(1-\lambda)\mu^-,
        \qquad \lambda\in[0,1].
\]

It remains to identify \(\lambda\). By definition of
\(\mathcal A_{\mathrm{anti}}\), there exist a finite-index subgroup
\(L\subset\mathbb Z^2\) and \(u\in\mathbb Z^2\) such that
\[
        \tau_\ell\eta=\eta \quad(\ell\in L),
        \qquad
        \tau_u\eta=\Theta\eta,
\]
where \(\Theta\) is the global spin-flip. Since the zero-field dynamics
commutes with translations and with \(\Theta\), one has, for every \(t\ge0\),
\[
        (\tau_u)_*(\delta_\eta P_t)=\Theta_*(\delta_\eta P_t).
\]
The same identity is preserved by Ces\`aro averaging and therefore holds for
every subsequential limit \(\mu\).

If
\[
        \mu=\lambda\mu^+ +(1-\lambda)\mu^-,
\]
then translation invariance of \(\mu^\pm\), together with
\[
        \Theta_*\mu^+=\mu^-,
        \qquad
        \Theta_*\mu^-=\mu^+,
\]
gives
\[
        \lambda\mu^+ +(1-\lambda)\mu^-
        =
        \lambda\mu^- +(1-\lambda)\mu^+.
\]
Since \(\mu^+\neq\mu^-\), this forces \(\lambda=1/2\). Therefore every
subsequential limit of \((\mu_T)\) is \(\frac12(\mu^++\mu^-)\), and hence
\[
        \mu_T \Rightarrow \frac12(\mu^++\mu^-).
\]

Thus the trajectories generated by
$\mathcal A_{\mathrm{anti}}$ remain pathwise confined in the centered
level, and hence outside the two equilibrium magnetization levels, while
their long-time Ces\`aro statistics converge to the symmetric Gibbs
mixture. This confinement therefore does not correspond to an additional
equilibrium phase.

\section{Conclusion}

For exact magnetization levels, deterministic-time confinement
coincides with pathwise confinement:
\[
D_r=S_{C_r}.
\]
This structural principle is not specific to the phase-coexistence
regime.

For the two-dimensional Ising model at inverse temperature
$\beta>\beta_c$, the equilibrium magnetization levels contain
full-measure, dense pathwise cores:
\[
\mu^+(S_{C_{m_\beta}})=1,
\qquad
\mu^-(S_{C_{-m_\beta}})=1.
\]
The centered level also contains a dense explicit core,
\[
\mathcal A_{\mathrm{anti}}
\subset S_{C_0}
\subset S_R,
\qquad
R=
\Omega\setminus
\bigl(C_{m_\beta}\cup C_{-m_\beta}\bigr).
\]
Thus equilibrium-typical and equilibrium-null confinement cores may
both be topologically dense.

The extremal Gibbs phases remain $\mu^+$ and $\mu^-$. The
centered confinement is a pathwise phenomenon, not a third Gibbs phase.
It remains open whether $S_R$ is comeagre and how the sets
$S_{C_r}$ can be characterized intrinsically, possibly through exact
empirical levels of local observables.

\section*{Acknowledgments}

The author thanks Jeffrey Steif for helpful comments and for raising the
question of how absorbing structures should be understood in the stochastic Ising model.


\begin{thebibliography}{10}

\bibitem{Aizenman80}
M.~Aizenman.
\newblock Translation invariance and instability of phase coexistence in the
  two-dimensional {I}sing system.
\newblock {\em Communications in Mathematical Physics}, 73(1):83--94, 1980.

\bibitem{Steif03}
I.~Benjamini, O.~H\"aggstr\"om, Y.~Peres, and J.~E.~Steif,
Which properties of a random sequence are dynamically sensitive?
\emph{Ann. Probab.} \textbf{31} (2003), no.~1, 1--34.

\bibitem{FriedliVelenik17}
S.~Friedli and Y.~Velenik.
\newblock {\em Statistical Mechanics of Lattice Systems: A Concrete
  Mathematical Introduction}.
\newblock Cambridge University Press, Cambridge, 2017.

\bibitem{Georgii11}
H.-O. Georgii.
\newblock {\em Gibbs Measures and Phase Transitions}, volume~9 of {\em De
  Gruyter Studies in Mathematics}.
\newblock De Gruyter, Berlin/Boston, 2 edition, 2011.

\bibitem{Harris72}
T.~E. Harris.
\newblock Nearest-neighbor markov interaction processes on multidimensional
  lattices.
\newblock {\em Advances in Mathematics}, 9(1):66--89, 1972.

\bibitem{Higuchi81}
Y.~Higuchi.
\newblock On the absence of non-translation invariant {G}ibbs states for the
  two-dimensional {I}sing system.
\newblock In {\em Random Fields, Vol. II}, volume~27 of {\em Colloquia
  Mathematica Societatis J{\'a}nos Bolyai}, pages 517--534. North-Holland,
  Amsterdam, 1981.

\bibitem{HolleyStroock77}
R.~A. Holley and D.~W. Stroock.
\newblock In one and two dimensions, every stationary measure for a stochastic
  ising model is a gibbs state.
\newblock {\em Communications in Mathematical Physics}, 55(1):37--45, 1977.

\bibitem{Liggett85}
T.~M. Liggett.
\newblock {\em Interacting Particle Systems}, volume 276 of {\em Grundlehren
  der Mathematischen Wissenschaften}.
\newblock Springer-Verlag, New York, 1985.

\bibitem{Lindenstrauss01}
E.~Lindenstrauss.
\newblock Pointwise theorems for amenable groups.
\newblock {\em Inventiones Mathematicae}, 146(2):259--295, 2001.

\bibitem{Rosenblatt71}
M.~Rosenblatt.
\newblock {\em Markov Processes: Structure and Asymptotic Behavior}, volume 184
  of {\em Grundlehren der Mathematischen Wissenschaften}.
\newblock Springer-Verlag, Berlin, 1971.

\end{thebibliography}
\end{document}